\documentclass[a4paper,english,12pt]{amsart}
\usepackage[cp1250]{inputenc}
\usepackage[margin=0.8in]{geometry}
\usepackage{amsfonts,amsmath,mathrsfs,amsthm}
\usepackage{topcapt}
\usepackage[T1]{fontenc}
\usepackage{tabularx}
\usepackage{graphicx}
\usepackage{hyperref}
\usepackage{multirow}
\usepackage{epsfig}
\usepackage[figuresright]{rotating}
\newtheorem{thm}{Theorem}[section]
\newtheorem{cor}[thm]{Corollary}
\newtheorem{lem}[thm]{Lemma}
\newtheorem{prop}[thm]{Proposition}
\newtheorem{f}[thm]{Fact}

\newtheorem{de}[thm]{Definiton}
\newtheorem{rem}[thm]{Remark}
\newtheorem*{con}{Conjecture}

\let\originalforall=\forall
\renewcommand{\forall}{\mathop{\vcenter{\hbox{\Large$\originalforall$}}}}

\let\originalexists=\exists
\renewcommand{\exists}{\mathop{\vcenter{\hbox{\Large$\originalexists$}}}}
\title{Spectrally reasonable measures II}
\author{Przemys\l aw Ohrysko}
\address{Chalmers University of Technology and the University of Gothenburg}
\email{p.ohrysko@gmail.com}
\thanks{Supported by foundations managed by The Royal Swedish Academy of Sciences}
\author{Micha\l \space Wojciechowski}
\address{Institute of Mathematics of Polish Academy of Sciences}
\email{miwoj.impan@gmail.com}
\begin{document}
\begin{abstract}
A measure on a locally compact Abelian group is said to have a natural spectrum if its spectrum is equal to the closure of the range of the Fourier-Stieltjes transform. In this paper we continue the study of spectrally reasonable measures (measures perturbing any measure with a natural spectrum to a measure with a natural spectrum) initiated in \cite{ow}. Particularly, we provide a full characterization of such measures for certain class of locally compact Abelian groups which includes the circle and the real line. We also elaborate on the spectral properties of measures with non-natural but real spectra constructed by F. Parreau.
\end{abstract}

\maketitle
\section{Introduction}
We collect first some basic facts from Banach algebra theory and harmonic analysis in order to fix the notation (our main reference for Banach algebra theory is \cite{z}, for harmonic analysis check \cite{r}). For a commutative unital Banach algebra $A$, the Gelfand space of $A$ (the set of all multiplicative-linear functionals endowed with weak$^{\ast}$-topology) will be abbreviated $\triangle(A)$ and the Gelfand transform of an element $x\in A$ is a surjection $\widehat{x}:\triangle(A)\rightarrow\sigma(x)$ defined by the formula: $\widehat{x}(\varphi)=\varphi(x)$ for $\varphi\in\triangle(A)$ where $\sigma(x):=\{\lambda\in\mathbb{C}:\mu-\lambda\delta_{0}\text{ is not invertible}\}$ is the spectrum of an element $x$. The spectral radius of an element $x\in A$ is denoted by $r(x)$. A Banach algebra with involution is called symmetric if $\sigma(x)\subset\mathbb{R}$ for every self-adjoint $x\in A$.
\\
Let $G$ be a locally compact Abelian group with its unitary dual $\widehat{G}$ and let $M(G)$ denote the Banach algebra of all complex-valued Borel regular measures equipped with the convolution product and the total variation norm. For $\mu\in M(G)$ we write $|\mu|$ to signify the total variation measure. The Fourier-Stieltjes transform will be treated as a restriction of the Gelfand transform to $\widehat{G}$. $M(G)$ is also equipped with involution $\mu\mapsto\widetilde{\mu}$ where $\widetilde{\mu}(E):=\overline{\mu(-E)}$ for every Borel set $E\subset G$. A measure $\mu$ is Hermitian if $\mu=\widetilde{\mu}$ or equivalently, if its Fourier-Stieltjes transform is real-valued. The closed ideal of measures with Fourier-Stieltjes transforms vanishing at infinity is denoted by $M_{0}(G)$, the closed subalgebra of all discrete (purely atomic) measures is denoted by $M_{d}(G)$ and the closed ideal of continuous (non-atomic) measures will be denoted by $M_{c}(G)$.
\\
In his fundamental paper (\cite{Zafran}) M. Zafran studied measures with a \textit{natural spectrum}.
\begin{de}
A measure $\mu\in M(G)$ is said to have a natural spectrum if $\sigma(\mu)=\overline{\widehat{\mu}(\widehat{G})}$. The set of all such measures will be denoted by $\mathcal{N}(G)$.
\end{de}
The existence of measures with a non-natural spectrum was first observed by N. Wiener and H. R. Pitt in \cite{wp} and in recognition of the authors this strange spectral behaviour was called the \textbf{Wiener--Pitt phenomenon}. The first rigorous proof of the existence of the Wiener--Pitt phenomenon was given by Y. Schreider \cite{schreider} and simplified later by Williamson \cite{wil}. An alternative approach (using Riesz products) was introduced by C.~C. Graham \cite{graham} and nowadays we know many examples of measures in $M_{0}(G)$ with a non-natural spectrum for any locally compact non-discrete Abelian group $G$ - consult for example Chapters 5--7 in \cite{grmc}.
\\
However, it is not difficult to list a broad class of measures in $\mathcal{N}(G)$, for example $L^{1}(G)\subset \mathcal{N}(G)$ and more generally $M_{00}(G)\subset \mathcal{N}(G)$ where
\begin{equation*}
M_{00}(G):=\{\mu\in M(G):\widehat{\mu}(\varphi)=0\text{ for all }\varphi\notin\widehat{G}\}\subset M_{0}(G).
\end{equation*}
It is also well-known that $M_{d}(G)\subset \mathcal{N}(G)$.
\\
Despite of the presence of the rich collection of examples it turned out that the set of measures with natural spectrum does not have good algebraic properties. In fact, it is  even not closed under addition which was first proved by M. Zafran in \cite{Zafran} with an additional assumption on the group. The general case was settled by O. Hatori in \cite{h} (for non-compact groups) and O. Hatori and E. Sato in \cite{hs} (for compact groups). We cite this two theorems for further reference.
\begin{thm}[Hatori]\label{hat}
Let $G$ be a non-compact locally compact Abelian group. Then $M(G)=\mathcal{N}(G)+L^{1}(G)$.
\end{thm}
\begin{thm}[Hatori,Sato]\label{hats}
Let $G$ be a compact Abelian group. Then $M(G)=\mathcal{N}(G)+\mathcal{N}(G)+M_{d}(G)$.
\end{thm}
Clearly, Theorems \ref{hat} and \ref{hats} imply that $\mathcal{N}(G)$ is not closed under addition unless $G$ is discrete as otherwise every measure would have a natural spectrum contradicting the existence of the Wiener--Pitt phenomenon.
\\
However, if we restrict our attention to measures in $M_{0}(G)$ for compact $G$ we have a positive result due to M. Zafran (\cite{Zafran}) on our disposal.
\begin{thm}[Zafran]\label{z}
Let $G$ be a compact Abelian group. Then $\mathcal{N}(G)\cap M_{0}(G)=M_{00}(G)$.
\end{thm}
It should be noted that no result of this type holds true for non-compact groups which follows easily from Hatori's theorem \ref{hat}.
\\
Since $\mathcal{N}(G)$ does not possess convenient algebraic structure we introduced in \cite{ow} the notion of \textit{spectrally reasonable measures}. As the present paper is a far reaching enhancement of the previous work we will refer to \cite{ow} frequently. Despite of the fact that the study in \cite{ow} was restricted to the circle group, the validity of basic results extends verbatim to arbitrary locally compact Abelian groups.
\begin{de}
Let $G$ be a locally compact Abelian group. A measure $\mu\in M(G)$ is said to be spectrally reasonable if $\mu+\nu\in \mathcal{N}(G)$ for every $\nu\in \mathcal{N}(G)$. The set of all such measures will be denoted by $\mathcal{S}(G)$.
\end{de}
We will cite two basic facts from \cite{ow} that are going to be used in the sequel (Lemma 6, Theorem 7 and Proposition 8 in the original paper).
\begin{lem}\label{rozl}
The spectrally reasonable measures have the following properties:
\begin{enumerate}
  \item If $\mu\in\mathcal{S}(G)$ and $\nu\in\mathcal{N}(G)$ then $\mu\ast\nu\in\mathcal{N}(G)$.
  \item If $\mu\in\mathcal{S}(G)$ and $\mu$ is invertible then $\mu^{-1}\in\mathcal{S}(G)$.
\end{enumerate}
\end{lem}
\begin{thm}\label{pod}
Let $G$ be a locally compact Abelian group. Then $\mathcal{S}(G)$ is a closed, unital, symmetric Banach$^{\ast}$ subalgebra of $M(G)$. Moreover, $\widehat{G}$ is dense in $\triangle(\mathcal{S}(G))$.
\end{thm}
Let $\mu\in M(G)$ and let $f\in L^{1}(|\mu|)$. We will write $f\mu$ for the measure in $M(G)$ defined via
\begin{equation*}
f\mu(E)=\int_{E}fd\mu.
\end{equation*}
By the Radon-Nikodym theorem every measure $\nu$ absolutely continuous with respect to $\mu$ ($\nu\ll\mu$) is of this form. An import case is $f=\gamma_{0}\in \widehat{G}$ as $\widehat{\gamma_{0}\mu}(\gamma)=\widehat{\mu}(\gamma-\gamma_{0})$ for $\gamma\in\widehat{G}$. The connection of absolute continuity and the operation $\mu\mapsto\gamma_{0}\mu$ will be explained in terms of $L$-subspaces.
\begin{de}
Let $G$ be a locally compact Abelian group and let $X\subset M(G)$ be a closed linear subspace (subalgebra, ideal). Then $X$ is called $L$-subspace ($L$-subalgebra, $L$-ideal) if for every $\mu\in X$ and $\nu\ll\mu$ we have $\nu\in X$.
\end{de}
Being an $L$-subspace is equivalent to the invariance under the operation $\mu\mapsto\gamma_{0}\mu$ as the following (well-known) fact shows whose proofs is included for the readers convenience.
\begin{f}\label{ly}
Let $X$ be a closed subspace of $M(G)$. Then $X$ is an $L$-subspace iff for every $\mu\in X$ and $\gamma\in\widehat{G}$ we have $\gamma\mu\in X$.
\end{f}
\begin{proof}
The first implication is clear, so let us assume that $X$ is a closed subspace such that for every $\mu\in X$ and $\gamma\in\widehat{G}$ we have $\gamma\mu\in X$. Suppose that there exists $\mu\in X$ and $f\in L^{1}(|\mu|)$ with $f\mu\notin X$. By Hahn-Banach theorem there is $\psi\in \left(L^{1}(|\mu|)\right)^{\ast}=L^{\infty}(|\mu|)$ satisfying $\psi|_{X\cap L^{1}(|\mu|)}=0$ and $\psi(f\mu)\neq 0$. It follows that
\begin{equation*}
\int_{G}\gamma\psi d|\mu|=0\text{ for every }\gamma\in\widehat{G}.
\end{equation*}
This implies, by the uniqueness theorem, that $\psi|\mu|=0$. Hence $\psi$ is a zero element of $L^{\infty}(|\mu|)$ contradicting $\psi(f\mu)\neq 0$.
\end{proof}
We will apply this fact to $\mathcal{S}(G)$.
\begin{prop}\label{lpod}
Let $G$ be a locally compact Abelian group. Then $\mathcal{S}(G)$ is an $L$-subalgebra of $M(G)$.
\end{prop}
\begin{proof}
In view of Fact \ref{ly} it is enough to show that for every $\mu\in\mathcal{S}(G)$ and $\gamma\in\widehat{G}$ we have $\gamma\mu\in\mathcal{S}(G)$. In order to prove this let us make the following observation: the mapping $\mu\mapsto\gamma_{0}\mu$ for a fixed $\gamma_{0}\in\widehat{G}$ is an automorphism of $M(G)$ preserving the image of the Fourier-Stieltjes transform. Therefore $\nu\in\mathcal{N}(G)$ iff $\gamma\nu\in\mathcal{N}(G)$ for every $\gamma\in\widehat{G}$ which finishes the proof of the proposition in an obvious way.
\end{proof}
In the light of the definition of the naturality of the spectrum and the prominent role of the class $M_{00}(G)$ it seems that the set of all measures $\mu\in M(G)$ such that $\widehat{\mu}(\varphi)=0$ for all $\varphi\notin\overline{\widehat{G}}$ should be of particular interest in this area. Quite surprisingly, relaying on the deep result from \cite{phg} these two sets coincide.
\begin{thm}
Let $\mu\in M(G)$. If $\widehat{\mu}(\varphi)=0$ for every $\varphi\notin\overline{\widehat{G}}$ then $\mu\in M_{00}(G)$.
\end{thm}
\begin{proof}
Let $X:=\{\varphi\in\triangle(M(G)):\widehat{\mu}(\varphi)=0\}\subset\triangle(M(G))$. By our assumption $\triangle(M(G))\setminus\overline{\widehat{G}}\subset X$. Clearly, $X$ is closed which gives $\overline{\triangle(M(G))\setminus\overline{\widehat{G}}}\subset X$. Let $\partial(M(G))$ denote the \v{S}ilov boundary\footnote{The \v{S}ilov boundary of a commutative Banach algebra $A$ is the smallest closed subset $E$ of $\triangle(A)$ such that $\sup\{|\widehat{x}(\varphi)|:\varphi\in E\}=r(x)$ for every $x\in A$} of $M(G)$. It is well-known that $\overline{\widehat{G}}\subset\partial(M(G))$ leading to
\begin{equation}\label{zaw}
\overline{\triangle(M(G))\setminus\partial(M(G))}\subset X \text{ and }\triangle(M(G))\setminus X\subset \triangle(M(G))\setminus\overline{\triangle(M(G))\setminus\partial(M(G))}.
\end{equation}
A standard formula from general topology implies
\begin{equation*}
\triangle(M(G))\setminus\overline{\triangle(M(G))\setminus\partial(M(G))}=\mathrm{Int}(\partial(M(G))).
\end{equation*}
However, by the titled result from \cite{phg} we have $\mathrm{Int}(\partial(M(G)))=\widehat{G}$. Now $(\ref{zaw})$ provides $\triangle(M(G))\setminus\widehat{G}\subset X$ and finishes the proof.

\end{proof}
\begin{rem}
As the problem of the naturality of the spectrum is meaningful only for non-discrete groups we always assume that $G$ is a locally compact non-discrete group.
\end{rem}
\section{Measures in $M_{0}$}
In this section we will prove the following dichotomy: $\mathcal{S}(G)\cap M_{0}(G)=\{0\}$ or $\mathcal{S}(G)\cap M_{0}(G)=M_{00}(G)$. The first case occurs for non-compact groups and the second for compact ones. We start with the latter case.
\subsection{Compact groups}
Let us recall first Lemma 2.2 from \cite{Zafran} (in fact we use only a part of the original lemma and the proof can be then vastly simplified - check for example Lemma 4.2 in \cite{owa}).
\begin{lem}\label{izo}
Let $G$ be a compact Abelian group and let $\mu\in M(G)$. If $\lambda$ is an isolated point of $\sigma(\mu)$ then $\lambda\in\widehat{\mu}(\widehat{G})$.
\end{lem}
Using Lemma \ref{izo} and Zafran's theorem (Theorem \ref{z}) we will prove now that $S(G)\cap M_{0}(G)=M_{00}(G)$. This result can be found in \cite{ow} but proved in greater generality with more sophisticated methods.
\begin{thm}\label{nato}
Let $G$ be a compact Abelian group. Then $\mathcal{S}(G)\cap M_{0}(G)=M_{00}(G)$.
\end{thm}
\begin{proof}
As $\mathcal{S}(G)\subset\mathcal{N}(G)$ the Zafran's theorem imply the inclusion $\mathcal{S}(G)\cap M_{0}(G)\subset M_{00}(G)$.
\\
For the reverse inclusion, let us take $\mu\in M_{00}(G)$, $\nu\in\mathcal{N}(G)$ and consider $\lambda\in\sigma(\mu+\nu)$. If $\lambda$ is an isolated point of $\sigma(\mu+\nu)$ then we are done by Lemma \ref{izo}. In case of $\lambda$ being an accumulation point of $\sigma(\mu+\nu)$ let us take a sequence of distinct complex numbers $\lambda_{k}\in\sigma(\mu+\nu)$ such that $\lambda_{k}\xrightarrow[k\rightarrow\infty]{}\lambda$ and let $\varphi_{k}\in\triangle(M(G))$ satisfy $\varphi_{k}(\mu+\nu)=\lambda_{k}$. Without loss of generality, we are allowed to assume $\varphi_{k}\notin\widehat{G}$ for every $k\in\mathbb{N}$. By the definition of $M_{00}(G)$ we have $\varphi_{k}(\mu+\nu)=\varphi_{k}(\nu)$ for every $k\in\mathbb{N}$. It follows that $\lambda$ is an accumulation point of $\sigma(\nu)=\overline{\widehat{\nu}(\widehat{G})}$. Let $\gamma_{n}\in\widehat{G}$ be a sequence of characters for which $\widehat{\nu}(\gamma_{n})\xrightarrow[n\rightarrow\infty]{}\lambda$. Then, as $\mu\in M_{00}(G)\subset M_{0}(G)$ we clearly have $(\widehat{\mu}+\widehat{\nu})(\gamma_{n})\xrightarrow[n\rightarrow\infty]{}\lambda$ and the argument is finished.
\end{proof}
\subsection{Non-compact groups}
We will deal first with absolutely continuous measures. To start the discussion we note that $\mathcal{S}(G)\cap L^{1}(G)$ is an $L$-subspace of $L^{1}(G)$ (compare with Proposition \ref{lpod}). $L$-subspaces of $L^{1}(G)$ has been characterized in \cite{s} leading to the following form of $\mathcal{S}(G)\cap L^{1}(G)$:
\begin{equation}\label{rozo}
\mathcal{S}(G)\cap L^{1}(G)=\chi_{H}L^{1}(G)\text{ for some measurable $H\subset G$}.
\end{equation}
We remark here that the set $H$ is defined up to a locally null set (with respect to the Haar measure on $G$) and all equalities and inclusions are meant in the same fashion.
\\
By Hatori's theorem (Theorem \ref{hat}), $H\neq G$. We will argue by contradiction now so let us suppose that $H$ is of positive Haar measure.
\\
It is straightforward to verify the statement: $f\in \mathcal{S}(G)\cap L^{1}(G)$ iff $\widetilde{f}\in\mathcal{S}(G)\cap L^{1}(G)$ implying $H=-H$. Moreover, as $\mathcal{S}(G)\cap L^{1}(G)$ is an algebra, we have $H+H\subset H$ by the characterisation of 'vanishing subalgebras' (here we use the result of T. Liu from \cite{l}, for greater generality check \cite{sh} and completely new approach in \cite{gh}). Now, by Steinhaus' theorem (check \cite{st}) $H$ is an open subgroup of $G$. There are two essentially different situations to be considered at this point.
\\
\underline{Case 1}: $H$ is non-compact.
\\
The proof in this case will be a simply application of the following lemma.
\begin{lem}\label{odwn}
Let $H$ be a closed subgroup of a locally compact Abelian group $G$ and let $\mu\in M(H)$. Then $\sigma_{M(H)}(\mu)=\sigma_{M(G)}(\mu)$\footnote{The first set is the spectrum of $\mu$ in $M(H)$ and the second is the spectrum of $\mu$ treated as an element in $M(G)$.}.
\end{lem}
\begin{proof}
As the inclusion $\sigma_{M(H)}(\mu)\subset\sigma_{M(G)}(\mu)$ is satisfied automatically, it is enough to verify the following statement: if $\mu\in M(H)$ is invertible in $M(G)$ then $\mu^{-1}\in M(H)$. To do so, let us take $\mu\in M(H)$ and $\nu\in M(G)$ such that $\mu\ast\nu=\delta_{0}$. We decompose $\nu$ relative to $H$: $\nu=\nu_{H}+\nu_{G\setminus H}$\footnote{For a Borel set $S\subset G$ and $\nu\in M(G)$, the measure $\nu_{S}$ is defined for any Borel set $E\subset G$ by the formula: $\nu_{S}(E):=\nu(S\cap E)$}. Then
\begin{equation*}
\mu\ast\nu_{H}+\mu\ast\nu_{G\setminus H}=\delta_{0}.
\end{equation*}
But $\mu\ast\nu_{G\setminus H}$ is supported on a set contained in $H+G\setminus H\subset G\setminus H$ and of course $\mu\ast\nu_{H}$ is concentrated on $H$ which gives $\mu\ast\nu_{G\setminus H}=0$. Now, by the uniqueness of the inverse we obtain $\mu^{-1}=\nu_{H}\in M(H)$.
\end{proof}
It is well-known that the image of the Fourier-Stieltjes transform of a measure $\mu\in M(H)\subset M(G)$ is a fixed subset of the complex plane, no matter if we treat $\mu$ as an element of $M(H)$ or $M(G)$. Applying Hatori's theorem (Theorem \ref{hat}) we find $\mu\in\mathcal{N}(H)$ and $f\in L^{1}(H)$ such that $\mu+f\notin\mathcal{N}(H)$. This, together with Lemma \ref{odwn} contradicts the assumption $L^{1}(H)\subset\mathcal{S}(G)$.
\\
\underline{Case 2}: $H$ is compact.
\\
Let $H^{\bot}:=\{\gamma\in\widehat{G}:\gamma(x)=1\text{ for every }x\in H\}\subset\widehat{G}$ be the anihilator of $H$ (for basic information on this notion consult section 2.1 in \cite{r}). Then $H^{\bot}\simeq\widehat{\left(G/H\right)}$. Since $H$ is open, $G/ H$ is discrete and hence the dual group is compact implying the compactness of $H^{\bot}$. On the other hand, $\widehat{G}/H^{\bot}\simeq \widehat{H}$ and so $H^{\bot}$ is open in $\widehat{G}$. Using general theory (check Theorem 41.5 and 41.15 in \cite{hr}), we will find a Helson set\footnote{A compact subset $K$ of $\widehat{G}$ is called a Helson set if for every $f\in C(K)$ there exists $h\in L^{1}(G)$ such that $\widehat{h}|_{K}=f$} $C\subset H^{\bot}$ homeomorphic to the Cantor set. By Alexandroff-Hausdorff theorem there exists a continuous surjection $h:C\rightarrow\overline{\mathbb{D}}:=\{z\in\mathbb{C}:|z|\leq 1\}$. As $C$ is a Helson set in $H^{\bot}$ there exists $f\in L^{1}(G)$ such that $\widehat{f}|_{C}=h$.
\\
Let us take any measure $\mu\notin\mathcal{N}(G)$ with $\|\mu\|=1$. Then $\mu\ast(\delta_{0}-\chi_{H})\notin\mathcal{N}(G)$. Indeed, $\mu\ast \chi_{H}\in L^{1}(G)\subset\mathcal{N}(G)$, so if $\mu\ast(\delta_{0}-\chi_{H})$ had a natural spectrum we would obtain $\mu\in\mathcal{N}(G)$ by means of the following lemma (check Lemma 20 in \cite{ow}).
\begin{lem}\label{mnoz}
Let $\mu,\nu\in\mathcal{N}(G)$ satisfy $\mu\ast\nu=0$. Then $\mu+\nu\in\mathcal{N}(G)$.
\end{lem}
Consider a measure $\nu\in M(G)$ defined as follows:
\begin{equation*}
\nu:=\mu\ast(\delta_{0}-\chi_{H})+f\ast\chi_{H}.
\end{equation*}
We will check that $\nu$ has a natural spectrum. To do so, let us take $\varphi\in\triangle(M(G))\setminus\widehat{G}$. Then $\varphi(\nu)=\varphi(\mu)\in\overline{\mathbb{D}}$ as $\|\mu\|=1$. But $\overline{\mathbb{D}}=h(C)=\widehat{f}|_{H^{\bot}}(C)=\widehat{\nu}(H^{\bot})$ which finishes the argument.
\\
Let us take a constant $c>0$ and analyze closely the spectrum of $\nu+c\mathbf{1}_{H}$. As $\mu\ast(\delta_{0}-\chi_{H})\notin\mathcal{N}(G)$ there exists $0\neq\lambda\in\sigma(\mu\ast(\delta_{0}-\chi_{H}))\setminus \overline{(\mu\ast(\delta_{0}-\chi_{H}))\widehat{\phantom{x}}(\widehat{G})}$ and $\varphi\in\triangle(M(G))\setminus\widehat{G}$ such that $\varphi(\mu\ast(\delta_{0}-\chi_{H}))=\varphi(\mu)=\lambda$. Recalling that $L^{1}(G)\subset M_{00}(G)$ we have $\varphi(\nu+c\chi_{H})=\varphi(\mu\ast(\delta_{0}-\chi_{H}))=\lambda$. Now, for $\gamma\in\widehat{G}\setminus H^{\bot}$ we obtain $(\nu+c\chi_{H})\widehat{\phantom{x}}(\gamma)=\widehat{\mu}(\gamma)$ and for $\gamma\in H^{\bot}$ we get $(\nu+c\chi_{H})\widehat{\phantom{x}}(\gamma)=\widehat{f}(\gamma)+c$. The definition of $\lambda$ gives $\lambda\notin\overline{(\nu+c\chi_{H})\widehat{\phantom{x}}(\widehat{G}\setminus H^{\bot})}$ and as $|\lambda|\leq \|\mu\|=1$ we can easily choose the constant $c$ big enough to obtain $\lambda\notin \overline{(\nu+c\chi_{H})\widehat{\phantom{x}}(H^{\bot})}$ finishing the proof of the following theorem.
\begin{thm}\label{abso}
Let $G$ be a non-compact and non-discrete locally compact Abelian group. Then $\mathcal{S}(G)\cap L^{1}(G)=\{0\}$.
\end{thm}

We are ready now to prove the main result of this section.
\begin{thm}
Let $G$ be a non-compact and non-discrete locally compact Abelian group. Then $\mathcal{S}(G)\cap M_{0}(G)=\{0\}$.
\end{thm}
\begin{proof}
Assume, towards contradiction, that there exists $0\neq \mu\in\mathcal{S}(G)\cap M_{0}(G)$. As $\mathcal{S}(G)\cap M_{0}(G)$ is clearly an $L$-subalgebra of $M(G)$ we are allowed to restrict the discussion to the case of a probability measure. Moreover, by Theorem \ref{abso} we have $\mu^{\ast n}\bot L^{1}(G)$ for every $n\in\mathbb{N}$. Since $\mu\in M_{0}(G)$ there exists a compact set $K\subset \widehat{G}$ such that $|\widehat{\mu}(\gamma)|<\varepsilon<\frac{1}{3}$. By Theorem 2.6.8 from \cite{r} we can find $k\in L^{1}(G)$ satisfying $\widehat{k}=1$ na $K$ and $\|k\|\leq 2$. Let us consider a measure $\nu$ defined below
\begin{equation*}
\nu:=\mu-\mu\ast k=\mu\ast(\delta_{0}-k).
\end{equation*}
Basing on the spectral radius formula (note that $\nu^{\ast n}=\mu^{\ast n}+h_{n}$ for some $h_{n}\in L^{1}(G)$) we get $r(\nu)\geq 1$. On the other hand, for $\gamma\in K$ we get $\widehat{\nu}(\gamma)=0$ and for $\gamma\notin K$ we obtain
\begin{equation*}
|\widehat{\nu}(\gamma)|=|\widehat{\mu}(\gamma)|\cdot |1-\widehat{k}(\gamma)|\leq \varepsilon\cdot \|\delta_{0}+k\|< 3\varepsilon<1.
\end{equation*}
Thus $\nu$ does not have a natural spectrum contradicting the assumption $\mu\in\mathcal{S}(G)$.
\end{proof}
Using once again the $L$-subspace property of $\mathcal{S}(G)\cap M_{0}(G)$ we obtain the following corollary.
\begin{cor}
Let $\mu\in\mathcal{S}(G)$ ($G$ is non-compact and non-discrete). Then $\mu^{\ast n}\bot M_{0}(G)$ for every $n\in\mathbb{N}$.
\end{cor}
\section{Measures with finite spectrum}
In this part we will prove that no measure with finite spectrum is reasonable unless it is a trigonometric polynomial (finite linear combination of characters) or '$\delta_{0} - $' is a trigonometric polynomial. We start with idempotent measures.
\begin{thm}\label{idem}
Let $\mu\in M(G)$ be an idempotent measure different from $0$ and $\delta_{0}$. If $\mu\in\mathcal{S}(G)$ then $\mu$ is a trigonometric polynomial or $\delta_{0}-\mu$ is a trigonometric polynomial. In both cases the group $G$ is compact.
\end{thm}
\begin{proof}
Let us consider $\mathrm{supp}\widehat{\mu}:=\{\gamma\in\widehat{G}:\widehat{\mu}(\gamma)=1\}$. Clearly, it is a closed-open subset of $\widehat{G}$. If $\mathrm{supp}\widehat{\mu}$ is finite then $\widehat{G}$ is discrete and $\mu$ is a trigonometric polynomial. If the complement of $\mathrm{supp}\widehat{\mu}$ is finite then $\delta_{0}-\mu$ is a trigonometric polynomial in the same way as before. Thus, we are allowed to assume that both $\mathrm{supp}\widehat{\mu}$ and its complement are infinite. By the famous Cohen's idempotent theorem $\mathrm{supp}\widehat{\mu}$ belongs to the open coset ring of $\widehat{G}$. By Lemma 6.9. from \cite{owa} $\mathrm{supp}\widehat{\mu}$ contains an open infinite coset possibly excluding finitely many elements. Hence $\mathrm{supp}\widehat{\mu}\supset \gamma_{0}L\setminus F$ for some open subgroup $L\subset\widehat{G}$, finite set $F\subset\widehat{G}$ and $\gamma_{0}\in\widehat{G}$. However, a brief look on the proof of Lemma 6.9. shows that the set $F$ may be non-empty iff $\widehat{G}$ is discrete. Then any finite set belongs to the coset ring and as trigonometric polynomials are spectrally reasonable (Theorem \ref{nato}) we can assume, without loss of generality, that $F=\emptyset$ (we consider a measure $\mu+f$ where $\widehat{f}=\chi_{F}$ which is in $\mathcal{S}(G)$ iff $\mu\in \mathcal{S}(G))$. As $\mathcal{S}(G)$ is an $L$-subalgebra we can further reduce the situation to $\mathrm{supp}\widehat{\mu}\supset L$. The general theory of anihilators and quotients gives $L=\widehat{\left(G/ L^{\bot}\right)}$. Let $bL$ be the Bohr compactification of $L$. Then $bL=\triangle(M_{d}(G/ L^{\bot}))$. As $L$ is infinite, $bL$ is an infinite compact group and so contains a Helson set $C$ homeomorphic to the Cantor set. By the same arguments as in the last section there exists a continuous surjection $c: C\rightarrow\overline{\mathbb{D}}$ and $\rho_{1}\in M_{d}(G/L^{\bot})$ such that $\widehat{\rho_{1}}|_{C}=c$. Using the standard surjection $F:M(G)\rightarrow M(G/H)$ defined via canonical epimorphism $\pi: G\rightarrow G/H$ (consult Theorem 2.7.2 in \cite{r}) we find $\rho\in M_{d}(G)$ satisfying $\widehat{\rho}|_{L}=\widehat{\rho_{1}}$.
\\
Let $\omega\in M(G)$ be any measure with non-natural spectrum and $\|\omega\|\leq 1$. Let us define a measure $\nu\in M(G)$ by the formula:
\begin{equation*}
\nu=\mu\ast\rho+(\delta_{0}-\mu)\ast\omega.
\end{equation*}
Recalling that $L\subset\mathrm{supp}\widehat{\mu}$ we get for $\gamma\in L$:
\begin{equation*}
\widehat{\nu}(\gamma)=\widehat{\rho}(\gamma)=\widehat{\rho_{1}}(\gamma)\text{ implying }\widehat{\rho_{1}}(L)\subset\widehat{\nu}(\widehat{G}).
\end{equation*}
But $L$ is dense in $bL$ which leads to
\begin{equation}\label{zawko}
\overline{\mathbb{D}}=c(C)=\widehat{\rho_{1}}(K)\subset\overline{\widehat{\rho_{1}}(L)}\subset\overline{\widehat{\nu}(\widehat{G})}.
\end{equation}
Let us take $\varphi\in\triangle(M(G))$ and consider two cases:
\begin{enumerate}
  \item $\varphi(\mu)=0$. Then $\varphi(\nu)=\varphi(\omega)\in\overline{\mathbb{D}}\subset\overline{\widehat{\nu}(\widehat{G})}$ by (\ref{zawko}).
  \item $\varphi(\mu)=1$. Then $\varphi(\nu)=\varphi(\mu\ast\rho)$. By our assumption $\mu\in\mathcal{S}(G)$ and as $\rho\in M_{d}(G)\subset\mathcal{N}(G)$ we get $\mu\ast\rho\in\mathcal{N}(G)$ by Lemma \ref{rozl}. Therefore $\varphi(\mu\ast\rho)\in\overline{\widehat{\mu\ast\rho}(\widehat{G})}\subset\overline{\widehat{\nu}(\widehat{G})}$ as if $\widehat{\mu}(\gamma_{0})=1$ for some $\gamma_{0}\in\widehat{G}$ then we have $\widehat{\nu}(\gamma_{0})=\widehat{\mu\ast\rho}(\gamma_{0})$ by the definition of $\nu$ (if $\widehat{\mu}(\gamma_{0})=0$ then the assertion is clear).
\end{enumerate}
Finally, $\nu\in\mathcal{N}(G)$. Since we have assumed that $\mu\in\mathcal{S}(G)$ we get $\nu\ast (\delta_{0}-\mu)=(\delta_{0}-\mu)\ast\omega\in\mathcal{N}(G)$. The whole argument can be repeated replacing $\mu$ with $\delta_{0}-\mu$ implying $\mu\ast\omega\in\mathcal{N}(G)$. However, by Lemma \ref{mnoz}, this results in $\omega\in\mathcal{N}(G)$ which is a desired contradiction.
\end{proof}
We will now draw several useful corollaries.
\begin{cor}\label{skon}
Let $\mu\in \mathcal{S}(G)\setminus\mathrm{lin}(\delta_{0})$ have finite spectrum. Then $\mu$ is a trigonometric polynomial or $\delta_{0}-\mu$ is a trigonometric polynomial. In both cases the group $G$ is compact.
\end{cor}
\begin{proof}
As $\sigma(\mu)$ is finite we use the functional calculus to obtain the decomposition $\mu=\lambda_{1}\nu_{1}+\lambda_{2}\nu_{2}+\ldots+\lambda_{n}\nu_{n}$ for some $\lambda_{k}\in\mathbb{C}$, $k=1,\ldots,n$ where $n=\#\sigma(\mu)$ and $\nu_{k}$ for $k\in\{1,\ldots,n\}$ are idempotent measures satisfying $\nu_{l}\ast\nu_{j}=0$ for $l\neq j$. Clearly, all $\nu_{k}'s$ are again spectrally reasonable so the assertion follows from Theorem \ref{idem}.
\end{proof}
Theorem \ref{idem} implies that there are no non-trivial idempotents in $\mathcal{S}(G)$ unless $G$ is compact. This statement, together with the \v{S}hilov idempotent theorem gives the second corollary.
\begin{cor}
Let $G$ be a locally compact non-compact group. Then $\triangle(\mathcal{S}(G))$ is connected.
\end{cor}
Since for $\mu\in \mathcal{S}(G)$ we have $\sigma_{\mathcal{S}(G)}(\mu)=\sigma(\mu)$ (consult Lemma \ref{rozl}) and the spectrum is an image of the Gelfand transform we get the last corollary of this section.
\begin{cor}\label{spoj}
Let $G$ be a locally compact non-compact group and let $\mu\in\mathcal{S}(G)$. Then $\sigma(\mu)$ is connected.
\end{cor}
\section{On Parreau measures}
In this section we analyze the spectral properties of measures constructed by F. Parreau in \cite{p}. The results of this part will not be used in the sequel (with the exception of Remark \ref{modp}).
\begin{de}
Let $G$ be a locally compact Abelian group. A Parreau measure on $G$ is a probability measure with real spectrum and all convolution powers singular with respect to the Haar measure on $G$.
\end{de}
We start with a more general class of measures.
\begin{prop}\label{zws}
Let $G$ be a compact Abelian group and let $\mu\in M_{0}(G)$ satisfy $\sigma(\mu)\subset\mathbb{R}$, $\|\mu\|=1$ and $\mu\notin\mathcal{N}(G)$. Then there exists $a,b\geq 0, a+b>0$ and a finite set $S$ of real numbers such that $\sigma(\mu)=I\cup S$ where $I=[-a,b]$.
\end{prop}
\begin{proof}
Let us prove first that $\sigma(\mu)$ contains a closed interval $J\ni 0$. Assume it is not the case. Then there exist two sequences of positive real numbers $(a_{k})_{k\in\mathbb{N}}$ and $(b_{n})_{k\in\mathbb{N}}$ convergent to $0$ with $-a_{k}<-a_{k+1}$, $b_{k}>b_{k+1}$ and $-a_{k},b_{k}\notin\sigma(\mu)$ for $k\geq 1$. Put $-a_{0}<-1$ and $b_{0}>1$. We have an obvious inclusion:
\begin{equation}\label{zawier}
\sigma(\mu)\subset (-a_{0},-a_{1})\cup\bigcup_{k=1}^{\infty}(-a_{k},-a_{k+1})\cup\left\{0\right\}\cup\bigcup_{k=1}(b_{k+1},b_{k})\cup (b_{1},b_{0})
\end{equation}
We will justify now that for every $k\in\mathbb{N}\cup\{0\}$ the set $\sigma(\mu)\cap (b_{k+1},b_{k})$ is finite (a similar statement on the sets $\sigma(\mu)\cap (-a_{k},-a_{k+1})$ can be proved analogously). It is enough to establish this result for $k=0$ as we will see that the argument for other $k's$ is identical. Put $A:=\sigma(\mu)\cap (-a_{0},b_{1})$ and $B:=\sigma(\mu)\cap (a_{1},a_{0})$. Then $\sigma(\mu)=A\cup B$, $A\cap B=\emptyset$ and we can find two disjoint open subsets $\widetilde{A}$ and $\widetilde{B}$ of the complex plane satisfying $\widetilde{A}\cap\sigma(\mu)=A$, $\widetilde{B}\cap\sigma(\mu)=B$ with $\sigma(\mu)\subset\widetilde{A}\cup\widetilde{B}$. Let us define $f:\widetilde{A}\cup\widetilde{B}\rightarrow\mathbb{C}$ by $f|_{\widetilde{A}}=0$ and $f|_{\widetilde{B}}=\mathrm{id}$. Consider $f(\mu)$ being an image of $\mu$ of the action of the functional calculus. Recalling that $\mu\in M_{0}(G)$ there are only finitely many $\gamma\in\widehat{G}$ for which $\widehat{\mu}(\gamma)\in (b_{1},b_{0})$ and hence $f(\mu)$ is a trigonometric polynomial. As $\mu=(\mathrm{id}-f)(\mu)+f(\mu)$ we obtain $\sigma(\mu)\cap (b_{1},b_{0})=\sigma(f(\mu))$. Thus $\sigma(\mu)\cap (b_{1},b_{0})$ is a finite set. By (\ref{zawier}), $\sigma(\mu)$ is countable as a countable union of finite sets. This is clearly a contradiction as $\sigma(\mu)$ has to be uncountable (check Lemma 2.6 in \cite{Zafran}).
\end{proof}
The following theorem shows that the spectrum of a Parreau measure on a compact Abelian group, despite of its very involved construction, has quite simple structure.
\begin{thm}
Let $G$ be a compact Abelian group and let $\mu\in M_{0}(G)$ be a Parreau measure. Then for some $c\in [0,1]$ we have $\sigma(\mu)=[-c,1]\cup S$ or $\sigma(\mu)=[-1,c]\cup S$ where $S$ is a finite subset of $\mathbb{R}$.
\end{thm}
\begin{proof}
Let $L:=\{\gamma\in\widehat{G}:|\widehat{\mu}(\gamma)|=1\}$. Since $\mu\in M_{0}(G)$, the set $L$ is finite. Let $\rho$ be a trigonometric polynomial such that $\widehat{\rho}(\gamma)=\widehat{\mu}(\gamma)$ for $\gamma\in L$ and $\widehat{\rho}(\gamma)=0$ for $\gamma\notin L$. Consider $\nu:=\mu-\rho$. Then, recalling once again that $\mu\in M_{0}(G)$, we get $\|\widehat{\nu}\|_{\infty}<1$. Using the spectral radius formula and observing that the singular part of $\nu^{\ast n}$ is equal to $\mu^{\ast n}$ for every $n\in\mathbb{N}$ we obtain $r(\nu)\geq 1$ showing $\nu\notin\mathcal{N}(G)$. As trigonometric polynomials are spectrally reasonable (compare Theorem \ref{nato}), it implies $\mu\notin\mathcal{N}(G)$. By Proposition \ref{zws}, $\sigma(\mu)=[-a,b]\cup S$ where $a,b\geq 0, a+b>0$ and $S$ is a finite set of real numbers.
\\
To justify the particular form of the interval $[-a,b]$, let us take $\varphi\in\triangle(M(G))\setminus\widehat{G}$ satisfying $|\varphi(\nu)|=1$. As $\sigma(\nu)\subset\mathbb{R}$ we in fact have $\varphi(\nu)\in\{-1,1\}$. Without loss of generality, let $\varphi(\nu)=1$. Clearly, $1\notin\overline{\widehat{\nu}(\widehat{G})}$ and by Lemma \ref{izo} we conclude that $1$ is not an isolated point of $\sigma(\nu)$. Thus, we find a sequence $(\varphi_{n})_{n\in\mathbb{N}}\subset\triangle(M(G))\setminus\widehat{G}$ for which $\varphi_{n}(\nu)\xrightarrow[n\rightarrow\infty]{}1$. But $L^{1}(G)\subset M_{00}(G)$ which gives $\varphi_{n}(\mu)\xrightarrow[n\rightarrow\infty]{}1$. Therefore, $\sigma(\mu)$ contains a sequence converging to $1$ which finishes the proof.
\end{proof}
\begin{rem}\label{modp}
The existence of Parreau measures allows a construction of measures with non-natural spectrum contained in the unit circle (consult the first part of the proof of Theorem 24 in \cite{ow}). Such measures will be called modified Parreau measures.
\end{rem}

\section{Measures with 'fat' Fourier-Stieltjes transforms}
The aim of this section is to prove that the Fourier-Stietljes transform of a spectrally reasonable measure cannot map compact subsets of the dual group to 'large' (in topological sense) subsets of the spectrum. This result implies that there are no non-trivial spectrally reasonable measures on $\mathbb{R}^{n}$.
\\
We start with two simple lemmas.
\begin{lem}\label{domk}
Let $\mu\notin\mathcal{N}(G)$. Then $\sigma(\mu)\setminus\overline{\widehat{\mu}(\widehat{G})}$ is not a closed subset of $\sigma(\mu)$.
\end{lem}
\begin{proof}
Since $\mu\notin\mathcal{N}(G)$ the set $\sigma(\mu)\setminus\overline{\widehat{\mu}(\widehat{G})}$ is non-empty. Suppose that $\sigma(\mu)\setminus\overline{\widehat{\mu}(\widehat{G})}$ is closed. Then we can easily find two disjoint open sets $U,V\subset\mathbb{C}$ such that $\sigma(\mu)\setminus\overline{\widehat{\mu}(\widehat{G})}\subset U$ and $\overline{\widehat{\mu}(\widehat{G})}\subset V$. Defining a function $f:U\cup V\rightarrow\mathbb{C}$ by the formula: $f|_{U}=1$ and $f|_{V}=0$ and applying the functional calculus results immediately in contradiction.
\end{proof}
\begin{lem}\label{glup}
Let $\mu,\nu\in M(G)$ satisfy $\mu\ast\nu=0$, $\mu+\nu\in\mathcal{N}(G)$ and $\#\sigma(\nu)\cap\widehat{\mu}\left(\triangle(M(G))\setminus\widehat{G}\right)<\infty$. Then $\mu\in\mathcal{N}(G)$.
\end{lem}
\begin{proof}
By Lemma \ref{domk} it is enough to show that $\sigma(\mu)\setminus\overline{\widehat{\mu}(\widehat{G})}$ is closed. Since the inclusion $\sigma(\mu)\setminus\overline{\widehat{\mu}(\widehat{G})}\subset \widehat{\mu}\left(\triangle(M(G))\setminus\widehat{G}\right)$ is obvious it is sufficient to justify $\sigma(\mu)\setminus\overline{\widehat{\mu}(\widehat{G})}\subset\sigma(\nu)$ as then the assertion follows from the last assumption of the lemma.
\\
Let $0\neq\lambda\in\sigma(\mu)\setminus\overline{\widehat{\mu}(\widehat{G})}$ and let us take $\varphi\in\triangle(M(G))\setminus \widehat{G}$ for which $\varphi(\mu)=\lambda$. Recalling that $\mu\ast\nu=0$ we obtain $\varphi(\nu)=0$. This leads to
\begin{equation*}
\varphi(\mu)=\varphi(\mu+\nu)\in\overline{\left(\widehat{\mu}+\widehat{\nu}\right)(\widehat{G})}\text{ as $\mu+\nu\in\mathcal{N}(G)$}.
\end{equation*}
Let $(\gamma_{n})_{n\in\mathbb{N}}$ be a sequence of characters satisfying $(\widehat{\mu}+\widehat{\nu})(\gamma_{n})\xrightarrow[n\rightarrow\infty]{}\lambda$. Since $\mu\ast\nu=0$, passing to a subsequence if necessary, we are allowed to assume that $\widehat{\nu}(\gamma_{n})\xrightarrow[n\rightarrow\infty]{}\lambda$ since otherwise $\lambda\in\overline{\widehat{\mu}(\widehat{G})}$. This gives $\lambda\in\sigma(\nu)$ and finishes the proof.
\end{proof}

The proof of the following result, sometimes called 'boundary bumping theorem', can be found in the Section 5 of \cite{n}.
\begin{thm}[Janiszewski]\label{cont}
Let $X$ be a non-trivial continuum\footnote{A non-trivial continuum is a compact connected topological space which is not a single point.}. and let $U\subset X$ be an open subset of $X$. Then $U$ contains a non-trivial continuum.
\end{thm}
\begin{thm}\label{glon}
Let $G$ be a locally compact non-compact Abelian group and let $\mu\in M(G)$ be a measure different than the constant multiple of $\delta_{0}$. If there exists a compact set $K\subset\widehat{G}$ such that $\widehat{\mu}(K)$ has non-empty interior (as a subspace of $\sigma(\mu)$) then $\mu\notin\mathcal{S}(G)$.
\end{thm}
\begin{proof}
Assume, towards contradiction, that $\mu\in\mathcal{S}(G)$. Let us consider a measure $\tau:=\mu\ast\widetilde{\mu}$. Since $\mathcal{S}(G)$ is a $\ast$-algebra we have $\tau\in\mathcal{S}(G)$ and $\sigma(\tau)\subset\mathbb{R}$ (consult Theorem \ref{pod}). By Corollary \ref{spoj}, $\sigma(\tau)$ is connected. Let $Y$ be a non-trivial continuum in $\widehat{\mu}(K)$ from Theorem \ref{cont} and let $K'$ be the preimage of $Y$. Then $K'\subset K$ is a compact subset of $\widehat{G}$. Now, there are two cases to consider:
\begin{enumerate}
  \item $\widehat{\tau}(K')$ is a non-degenerate interval.
  \\
  Here we apply a function $x\mapsto e^{icx}$ for sufficiently big $c>0$ to obtain a measure $\nu\in\mathcal{S}(G)$ with $\widehat{\nu}(K')=\{z\in\mathbb{C}:|z|=1\}$
  \item $\widehat{\tau}(K')$ is a single point.
  \\
  We are allowed to assume that $\widehat{\tau}(K')=\{1\}$. Then, $\widehat{\mu}(K')$ is a non-trivial subcontinuum of the unit circle which is obviously an arc. Raising $\mu$ to a sufficiently big convolution power we obtain the same conclusion as in the previous point.
\end{enumerate}
Therefore, without loss of generality, we are allowed to assume that we have on our disposal a measure $\nu\in\mathcal{S}(G)$ satisfying $\widehat{\nu}(K')=\{z\in\mathbb{C}:|z|=1\}$ where $K'$ is a compact subset $\widehat{G}$.
\\
Using Theorem 2.6.8 (check also Theorem 2.6.1) from \cite{r} we find $f\in L^{1}(G)$ with compactly supported Fourier transform such that $\widehat{f}=1$ on $K'$, $\widehat{f}\geq 0$ and a function $g\in L^{1}(G)$ for which $\widehat{g}=1$ on the support of the Fourier transform of $f$. Let $\omega\in M(G)$ be measure with non-natural spectrum contained in the unit circle (modified Parreau measure - check Remark \ref{modp}) and define $\rho\in M(G)$ by the formula:
\begin{equation*}
\rho:=f\ast\nu+(\delta_{0}-g)\ast\omega.
\end{equation*}
We verify that $\rho\in\mathcal{N}(G)$: let us take $\varphi\in\triangle(M(G))\setminus\widehat{G}$ and note that $\varphi(\rho)=\varphi(\omega)\in\{z\in\mathbb{C}:|z|=1\}$ (recall $\varphi|_{L^{1}(G)}=0$). Now, for $\gamma\in K'$ we have $\widehat{\rho}(\gamma)=\widehat{\nu}(\gamma)$ implying $\{z\in\mathbb{C}:|z|=1\}=\widehat{\nu}(K')=\widehat{\rho}(K')\subset\widehat{\rho}(\widehat{G})$ which finishes the argument.
\\
By the assumption, $\nu\in\mathcal{S}(G)$ which gives $\nu^{-1}\in\mathcal{S}(G)$ (compare Lemma \ref{rozl}). Thus $\rho\ast\nu^{-1}\in\mathcal{N}(G)$,
\begin{equation*}
\rho\ast\nu^{-1}=f+(\delta_{0}-g)\ast\omega\ast\nu^{-1}\in\mathcal{N}(G).
\end{equation*}
Put $\xi:=(\delta_{0}-g)\ast\omega\ast\nu^{-1}$ and let us check the assumptions of Lemma \ref{glup}:$f\ast (\delta_{0}-g)=0$ implying $f\ast\xi=0$ and since $\xi+f\rho\ast\nu^{-1}$ we also get $\xi+f\in\mathcal{N}(G)$. As $\sigma(f)\subset\mathbb{R}_{+}$ and for $\varphi\in\triangle(M(G))\setminus\widehat{G}$ we get $\varphi(\xi)=\varphi(\omega)\varphi(\nu^{-1})\in\{z\in\mathbb{C}:|z|=1\}$ we conclude $\sigma(f)\cap\widehat{\xi}\left(\triangle(M(G))\setminus\widehat{G}\right)\subset\{1\}$. By Lemma \ref{glup} we obtain $\xi\in\mathcal{N}(G)$. Using the properties of spectrally reasonable measures (check Lemma \ref{rozl}), $(\delta_{0}-g)\ast\omega\in\mathcal{N}(G)$. This leads to a contradiction as for $\varphi\in\triangle(M(G))\setminus\widehat{G}$ such that $\varphi(\omega)\notin\overline{\widehat{\omega}(\widehat{G})}$ we get $\varphi((\delta_{0}-g)\ast\omega)=\varphi(\omega)\in\{z\in\mathbb{C}:|z|=1\}$ so by the assumption on the naturality of the spectrum of $(\delta_{0}-g)\ast\omega$ there exists a sequence of characters $(\gamma_{n})_{n\in\mathbb{N}}$ satisfying $(1-\widehat{g}(\gamma_{n}))\cdot\widehat{\omega}(\gamma_{n})\xrightarrow[n\rightarrow\infty]{}\varphi(\omega)$ and this implies $\widehat{\omega}(\gamma_{n})\xrightarrow[n\rightarrow\infty]{}\varphi(\omega)$ since $\sigma(\omega)\subset \{z\in\mathbb{C}:|z|=1\}$.
\end{proof}
\begin{thm}\label{spoj1}
Let $G$ be a locally compact Abelian group such that $\widehat{G}$ is connected. Then $\mathcal{S}(G)=\mathrm{lin}(\delta_{0})$.
\end{thm}
\begin{proof}
Let $V\subset \widehat{G}$ be an open symmetric ($V=-V$) neighborhood of the unit with compact closure and let us define $V_{1}:= V$ and $V_{n+1}:=V_{n}+V$ for $n\geq 1$. Consider
\begin{equation*}
H:=\bigcup_{n\in\mathbb{N}}V_{n}.
\end{equation*}
It is clear that $H$ is an open subgroup of $\widehat{G}$ and hence also closed. By the connectedness of $\widehat{G}$ we get $H=\widehat{G}$. Moreover, $\overline{V_{n}}$ is compact for every $n\in\mathbb{N}$ (easy exercise in general topology).
\\
Let $\mu\in\mathcal{S}(G)$ and define $\tau=\mu\ast\widetilde{\mu}\in\mathcal{S}(G)$. Then $\widehat{\tau}(\widehat{G})$ is a connected subset of $\mathbb{R}_{+}$. If $\widehat{\tau}(\widehat{G})$ is a non-degenerate interval then an application of the Baire category theorem shows that there exists $n_{0}\in\mathbb{N}$ for which $\widehat{\tau}\left(\overline{V_{n_{0}}}\right)$ has a non-empty interior in $\overline{\widehat{\tau}(\widehat{G})}=\sigma(\tau)$. This leads to a contradiction by Theorem \ref{glon}.
\\
If $\widehat{\tau}(\widehat{G})$ is a single point then taking a suitable constant multiple of $\tau$ we are allowed to assume $\widehat{\tau}(\widehat{G})=\{1\}$ and hence $\widehat{\mu}(\widehat{G})$ is a connected subset of the unit circle. Now, in case of a one-point image we obtain $\mu\in\mathrm{lin}(\delta_{0})$ and for the other case we can repeat the Bair category argument and apply Theorem \ref{glon} to obtain a contradiction.
\end{proof}
By the structure theorem for locally compact Abelian groups, if $\widehat{G}$ is connected then $\widehat{G}=\mathbb{R}^{n}\times K$ where $K$ is a connected compact Abelian group. However, by Theorem 2.5.6 in \cite{r}, the aforementioned form of $\widehat{G}$ is equivalent to $G=\mathbb{R}^{n}\times D$ where $D$ is a discrete torsion-free group. In view of this facts Theorem \ref{spoj1} is equivalent to the following one.
\begin{thm}\label{spojk}
Let $G=\mathbb{R}^{n}\times D$ where $D$ is a discrete torsion-free group. Then $\mathcal{S}(G)=\mathrm{lin}(\delta_{0})$.
\end{thm}
\section{The circle group}
In this section we characterize all spectrally reasonable measures on the circle group.
\\
We start with recalling the fact on discrete measures (see Section 4 in \cite{ow}).
\begin{prop}\label{dysnie}
$\mathcal{S}(\mathbb{T})\cap M_{d}(\mathbb{T})=\mathrm{lin}(\delta_{0})$.
\end{prop}
\begin{proof}
We will only show how to simplify the original argument from \cite{ow}. Since $\mathcal{S}(\mathbb{T})$ is an $L$-subalgebra of $M(\mathbb{T})$ is it enough to justify $\delta_{\alpha}\notin\mathcal{S}(\mathbb{T})$ for every $\alpha\in\mathbb{T}\setminus\{0\}$. If $\alpha$ is an element of finite order in $\mathbb{T}$ then the result follows from Corollary \ref{skon} or Theorem 21 in \cite{ow}. In case of an element of infinite order we can refer to Theorem 24 in \cite{ow} but let us provide a shorter argument for reader's convenience.
\\
Let $\alpha\in\mathbb{T}$ be an element of infinite order and assume, towards contradiction, that $\delta_{\alpha}\in\mathcal{S}(\mathbb{T})$. Let $\mu\in M(\mathbb{T})$ be the modified Parreau measure (measure with non-natural spectrum contained in the unit circle, check Remark \ref{modp}). Consider $\nu\in\ M(\mathbb{T})$ defined by the formula:
\begin{equation*}
\nu:=\frac{\delta_{0}+\delta_{\pi}}{2}\ast\delta_{\alpha}+\frac{\delta_{0}-\delta_{\pi}}{2}\ast\mu.
\end{equation*}
The closure of the range of the Fourier-Stieltjes transform of the first summand fills the whole unit circle and since the supports of the Fourier-Stieltjes transforms of both summands are disjoin we easily get $\nu\in\mathcal{N}(\mathbb{T})$. By our assumption, $\nu-\delta_{\alpha}\in\mathcal{N}(\mathbb{T})$ and of course
\begin{equation}\label{p}
\nu-\delta_{\alpha}=-\frac{\delta_{0}-\delta_{\pi}}{2}\ast(\delta_{\alpha}-\mu)\in\mathcal{N}(\mathbb{T}).
\end{equation}
The same argument applied to a measure $\rho:=\frac{\delta_{0}-\delta_{\pi}}{2}\ast\delta_{\alpha}+\frac{\delta_{0}+\delta_{\pi}}{2}\ast\mu$ gives
\begin{equation}\label{d}
\rho-\delta_{\alpha}=-\frac{\delta_{0}+\delta_{\pi}}{2}\ast(\delta_{\alpha}-\mu)\in\mathcal{N}(\mathbb{T}).
\end{equation}
Combining (\ref{p}), (\ref{d}) with Lemma \ref{mnoz} we obtain $\delta_{\alpha}-\mu\in\mathcal{N}(\mathbb{T})$ and consequently $\mu\in\mathcal{N}(\mathbb{T})$ which is a contradiction.
\end{proof}
\begin{thm}\label{okz}
$\mathcal{S}(\mathbb{T})=M_{00}(\mathbb{T})\oplus\mathrm{lin}(\delta_{0})$.
\end{thm}
\begin{proof}
Let $\mu\in\mathcal{S}(\mathbb{T})$ satisfy $\|\mu\|\leq 1$. Since $\mathcal{S}(\mathbb{T})$ is an $L$-subalgebra of $M(\mathbb{T})$ we can assume, in view of Proposition \ref{dysnie}, that $\mu\in M_{c}(\mathbb{T})$. Put $\tau:=\mu\ast\widetilde{\mu}\in\mathcal{S}(\mathbb{T})\cap M_{c})(\mathbb{T})$ (compare Theorem \ref{pod}). The characterization of spectrally reasonable measures with finite spectrum established in Corollary \ref{skon} allows us to assume that $\sigma(\tau)$ is infinite.
\\
If $\sigma(\tau)$ does not contain any interval then let us take a sequence $\varepsilon_{n}\xrightarrow[n\rightarrow\infty]{}0$ such that $\sigma(\tau)\cap (\varepsilon_{n})_{n=1}^{\infty}=\emptyset$. The set $[0,1]\setminus\sigma(\tau)$ is open (as a subset of $[0,1]$) and hence for every $n\in\mathbb{N}$, it contains an open interval with $\varepsilon_{n}$ in the interior. Let us fix $n_{0}\in\mathbb{N}$ and find two disjoint open subsets $U,V\subset\mathbb{C}$ such that $[0,\varepsilon_{n})\subset U$ and $(\varepsilon,1]\subset V$. Consider $f:U\cup V\rightarrow\mathbb{C}$ defined by the conditions $f|_{U}=0$, $f|_{V}=1$. The application of functional calculus gives an idempotent measure $\rho:=f(\mu)$ for which $\mathrm{supp}\widehat{\rho}=\{k\in\mathbb{Z}:|\widehat{\tau}(k)|>\varepsilon_{n}\}$. By Theorem \ref{idem}, the set $\mathrm{supp}\widehat{\rho}$ is finite or its complement is finite. There are two cases to be considered at this point.
\begin{enumerate}
  \item For every $n\in\mathbb{N}$ the set obtained from the described procedure is finite.
  \\
  Then $\tau\in M_{0}(\mathbb{T})$ implying $\mu\in M_{0}(\mathbb{T})$. Of course, we also have $\mu\in\mathcal{N}(\mathbb{T})$. Using Zafran's theorem we conclude $\mu\in M_{00}(\mathbb{T})$.
  \item For some $n_{0}\in\mathbb{N}$ the set $\{k\in\mathbb{Z}:|\widehat{\tau}(k)|>\varepsilon_{n_{0}}\}$ has a finite complement. But then, by Wiener's lemma, the measure $\tau$ has a non-trivial discrete part which is a contradiction.
\end{enumerate}
Thus, we can assume that $\sigma(\tau)$ contains a non-degenerate closed interval $I$. Then $I=\left(\overline{\widehat{\tau}(2\mathbb{Z})}\cap I\right)\cup\left(\overline{\widehat{\tau}(2\mathbb{Z}+1)}\cap I\right)$. At least one of these sets contains a non-degenerate interval $J$ (it follows from Baire's theorem). Using a linear transformation we reduce the situation to the case of $J=[0,2\pi]$.
\\
Let $\omega\in M(\mathbb{T})$ be a modified Parreau measure (see Remark \ref{modp}). Note that $\omega=\frac{\delta_{0}+\delta_{\pi}}{2}\ast\omega+\frac{\delta_{0}-\delta_{\pi}}{2}\ast\omega$. Since $\frac{\delta_{0}+\delta_{\pi}}{2}\ast \frac{\delta_{0}-\delta_{\pi}}{2}=0$ we have $\frac{\delta_{0}+\delta_{\pi}}{2}\ast\omega\notin\mathcal{N}(\mathbb{T})$ or $\frac{\delta_{0}-\delta_{\pi}}{2}\ast\omega\notin\mathcal{N}(\mathbb{T})$ by Lemma \ref{mnoz}. Without loss of generality, we assume $\frac{\delta_{0}-\delta_{\pi}}{2}\ast\omega\notin\mathcal{N}(\mathbb{T})$.
\\
Let us return to the measure $\tau$. If $J=[0,2\pi]\subset \overline{\widehat{\tau}(2\mathbb{Z})}$ then we proceed and in the second case we replace the measure $\tau$ by $e^{it}\tau$ reducing the discussion to the first case.
\\
Let us define $\xi\in M(\mathbb{T})$ by the formula:
\begin{equation*}
\xi:=\frac{\delta_{0}+\delta_{\pi}}{2}\ast e^{i\tau}+\frac{\delta_{0}-\delta_{\pi}}{2}\ast\omega.
\end{equation*}
We verify $\xi\in\mathcal{N}(\mathbb{T})$ which implies $\xi\ast e^{-i\tau}=\frac{\delta_{0}+\delta_{\pi}}{2}+\frac{\delta_{0}-\delta_{\pi}}{2}\ast\omega\ast e^{-i\tau}\in\mathcal{N}(\mathbb{T})$ (consult Lemma \ref{rozl}). It is elementary to check the assumptions of Lemma \ref{glup} which gives $\frac{\delta_{0}-\delta_{\pi}}{2}\ast\omega\ast e^{-i\tau}\in\mathcal{N}(\mathbb{T})$ and finally $\frac{\delta_{0}-\delta_{\pi}}{2}\ast\omega\in\mathcal{N}(\mathbb{T})$ again by Lemma \ref{rozl} which is the desired contradiction.
\end{proof}
\section{Concluding remarks}
The approach used in the proof of Theorem \ref{okz} can be aaplied to prove an analogous assertion for a wider class of compact Abelian groups. The inspection of the arguments utilized in the reasoning shows that the whole justification relies only on the existence of a subgroup of finite index in the dual group group of $G$. However, there are compact Abelian groups for which this condition is not satisfied - a standard examples are given by the groups of $p$-adic integers.
\\
In view of the results of the second section, Theorems \ref{spojk} and \ref{okz} we strongly believe that the following conjecture holds true.
\pagebreak
\begin{con}
Let $G$ be a locally compact non-discrete Abelian group.
\begin{enumerate}
  \item If $G$ is compact then $\mathcal{S}(G)=M_{00}(G)\oplus\mathrm{lin}(\delta_{0})$.
  \item If $G$ is non-compact then $\mathcal{S}(G)=\mathrm{lin}(\delta_{0})$.
\end{enumerate}
\end{con}

\end{document}